\documentclass[journal, twoside, web]{ieeetran}  

\IEEEoverridecommandlockouts                              

\usepackage{graphics} 
\usepackage{epsfig} 
\usepackage{mathptmx} 
\usepackage{times} 
\usepackage{amsmath} 
\usepackage{amssymb}  
\usepackage{graphicx}
\usepackage{relsize}
\usepackage[normalem]{ulem}
\DeclareMathAlphabet{\mathcal}{OMS}{cmsy}{m}{n}
\SetMathAlphabet{\mathcal}{bold}{OMS}{cmsy}{b}{n}

\title{\LARGE \bf
{On cost design in applications of optimal control}
}

\author{Taouba Jouini and Anders Rantzer
\thanks{*This work has received funding from the European Research Council (ERC) under the European Union's Horizon 2020 research  and innovation program (grant agreement No: 834142) and the Swedish Research Council (grant 2019-0069).}%
\thanks{The authors are with the Department of Automatic Control, LTH, Lund University, Lund, Sweden. E-mails:
        \tt\small \{taouba.jouini, anders.rantzer\}@control.lth.se.}}%

\graphicspath{{./fig/}}
\usepackage[usenames,dvipsnames]{color}
\usepackage{epstopdf}

\usepackage{amsmath,amsthm,amssymb,mathrsfs,dsfont}
\usepackage{amsfonts}
\usepackage{siunitx}
\usepackage[colorinlistoftodos,prependcaption,textsize=tiny]{todonotes}
\usepackage{tikz}
\usetikzlibrary{matrix}

\makeatletter
\newcommand*{\rom}[1]{\expandafter\@slowromancap\romannumeral #1@}
\makeatother

\usepackage{cite}

\newcommand\oprocendsymbol{\hbox{$\blacksquare$}}

\newcommand{\dd}[0]{\mathrm d}
\newcommand\oprocend{\relax\ifmmode\else\unskip\hfill\fi\oprocendsymbol}

\newcommand{\real}[0]{\mathbb R}

\providecommand{\norm}[1]{\lVert#1\rVert}

\newtheorem{theorem}{Theorem}[section]

\newtheorem{corollary}[theorem]{Corollary}
\newtheorem{proposition}[theorem]{Proposition}

\newtheorem{remark}[]{Remark}
\newtheorem{assumption}[]{Assumption}

\usepackage{verbatim} 

\usepackage[normalem]{ulem}
\newcommand\rout{\bgroup\markoverwith{\textcolor{red}{/}}\ULon} 

\usepackage[prependcaption,colorinlistoftodos]{todonotes}

\begin{document}
\maketitle
\thispagestyle{empty}

\begin{abstract}
A new approach to feedback control design based on optimal control is
proposed. Instead of expensive computations of the value function for
different penalties on the states and inputs, we use a control Lyapunov function that amounts to be a value function of an optimal control problem with suitable cost design and then study combinations of input and state penalty that are
compatible with this value function. This drastically simplifies the
role of the Hamilton-Jacobi-Bellman equation, since it is no longer a
partial differential equation to be solved, but an algebraic
relationship between different terms of the cost. The paper
illustrates this idea in different examples, including $\mathcal{H}_{\infty}$
control and optimal control of coupled oscillators.

\end{abstract}
 
\begin{IEEEkeywords}
Optimal control, Stability of nonlinear systems, Lyapunov methods
\end{IEEEkeywords}

\section{Introduction}

 \IEEEPARstart{T}{he} objective in optimal control problems is to transfer the state of a dynamical system with minimum cost from one point to another. The advent of modern control theory, particularly the formulation of the
famous Maximum Principle of Pontryagin \cite{9186331} has had a considerable impact on the treatment of optimization theory.
Dynamic programming gives necessary and sufficient conditions for optimality and optimal control laws in feedback form, which are very satisfactory but suffer from several drawbacks \cite{bertsekas2011dynamic, liberzon2011calculus}. First, analytic solutions can only be obtained in few cases (in particular linear quadratic problems). Second, the Hamilton-Jacobi-Bellman (HJB) partial differential equation (PDE) is in general very hard to solve numerically. The main problem is that the full state space must be discretized and a huge number of samples are needed to get reasonable solutions. This is {\em the curse of dimensionality}. 
For this, many efforts have been dedicated to find solutions of value function for HJB-PDE, either numerically \cite{lukes1969optimal} or by relaxing the equality to inequality using approximate dynamic programming \cite{powell2007approximate}.
 
The traditional way to use optimal control is to view the cost function as a set of tuning knobs that can be used to influence the trade-off between control effort and error decay rates. This works well in idealized settings such as linear quadratic control, but for nonlinear problems the map from cost function to the optimal controller could be overwhelmingly complicated. The purpose of this paper is to show that by carefully restricting the choice of the cost function, a simple map from parameters in the cost function to an explicit expression for the optimal controller can be obtained also for nonlinear systems.
In fact, our analysis provides a novel perspective for the application
of optimal control in engineering systems and makes a significant twist
compared to the classical approach. The idea is that, once a stabilizing feedback controller with a (control) Lyapunov function is found, then by appropriate choice of the cost function, involving state and
input penalties, the control Lyapunov function satisfies the HJB equation and is a value function of the optimal control problem.  As a consequence, a whole family of other cost functions will fit as well for
different penalties on the states and inputs. This makes it possible to design stabilizing controllers that are uniquely optimal for nonlinear systems
in a manner comparable to linear quadratic control for linear systems.
Our approach keeps a simple structure of the cost for nonlinear systems, while adding suitable parametrization and thus circumvents the computational complexity related to solving for a value function by suggesting a fixed (control) Lyapunov function a priori.
For this, we showcase the role the cost design plays in two typical settings of optimal control problems: first for nominal or disturbance-free and second for disturbance attenuation or robust $\mathcal{H}_{\infty}$ optimal control \cite{scherer2001theory, zhou1998essentials,bacsar2008h}. Finally,  we clarify our results with examples related to classical equations in linear and nonlinear control theory. As a continuation of ideas from \cite{jouini2021optimal}, we opt for an application to coupled oscillators that can represent for e.g. controlled inverters in power systems. 


The paper unfurls as follows: Section \ref{sec: cost-des} motivates and provides the main result on cost design for the nominal and disturbance attenuation case. Section~\ref{sec: application} applies our theory to coupled oscillators with numerical simulations.

{\em Notation:} 
Let $\mathds{1}_n$ denote the column vector of all ones and $I_n$ the $n-$th dimensional identity matrix.  We denote by $P>~0$ a symmetric and positive definite matrix and $\real_{>0}$ be the set of positive real numbers. Let $\norm{\cdot}_{P}=\sqrt{(\cdot)^\top P \,(\cdot)}$. Given a vector $v$, let  $\norm{v}_{\infty}=\sup_{i=1\dots n} \vert v_i\vert$, $\underline{\sin}(v)$ and $\underline{\cos}(v)$ be the vector-valued sine and cosine functions. Given a differentiable function $V(x)$, let $\nabla_x V=\frac{\partial V}{\partial x}$ be the the gradient of $V$ at $x$ and $\nabla^2_x V$ is the Hessian of $V$ at x. Given a matrix $A$, let $\text{Im}(A)$ denote its image space.
Consider a connected undirected graph $G=(\mathcal{V}, \mathcal{E})$ consisting of $\vert\mathcal{V}\vert=n$ nodes  and $\vert\mathcal{E}\vert=m$ edges. By assigning an arbitrary orientation to the $m$ edges, the incidence matrix $\mathcal{B}\in\real^{n\times m}$ is
defined element-wise as $\mathcal{B}_{il} = 1$, if node $i$ is the sink of the $l-$th edge, $\mathcal{B}_{il}=-1$ if $i$ is the source of the $l-$th edge, and $\mathcal{B}_{il} = 0$ otherwise. We denote by $\mathcal{N}_i$ the neighbor set of node~$i\in\mathcal{V}$. 

\section{Main result} 
\label{sec: cost-des}
We start our analysis with the following motivating example.

{\em Example 1:} For the matrices $R=R^\top>0$, consider the following nonlinear optimal control problem 
\begin{align*}
 V(x_0):=&\min_u \int^\infty_0 \left(q(x(s))+\norm{u(s)}^2_R\right)\dd s, \\
 &\text{s.t. } \dot x= -\underline{\sin}(x)+u, \quad x(0)=x_0, \\
 &x\in\mathcal{X}=\{x\in\real^n:\{\norm{x}_\infty <\pi/2: \mathds{1}^\top \underline{\cos}(x)\geq c\}, \nonumber
\end{align*}
for some $0<c<n$, for two different cases,
\begin{align*}
  &\textbf{Case 1:}& q(x)&=\norm{x}^2,\\
  &\textbf{Case 2:}& q(x)&=\norm{\underline{\sin}(x)}^2_{I_n+R^{-1}/4}.
\end{align*}
The first case may look simpler on the surface, since the cost function {is quadratic in $x$}. However, a closer look at Case~1 leads to a HJB partial differential equation, that is difficult to solve. At the same time, as we will see in the remainder, Case~2 is a special case of a rich class of problems that have a simple explicit solution. In fact, the optimal control law is given by
\begin{align*}
  u^*(x)&={-\frac{1}{2}} R^{-1}\underline{\sin}(x),
\end{align*}
and the value function amounts to
\begin{align*}
  V(x)=-\mathds{1}^\top_n \underline{\cos}(x)+n, 
\end{align*}
{To see that $V$ is positive definite, note that $V(0)=0$ and $V(x)$ is strictly convex and thus positive definite ($\nabla^2_x V(x)>0$ for all $x\in\mathcal{X}$).}
Notice that the matrix $R$ appears in the expression for $q$, but not in the value function $V$. Hence, when the penalty on the input is increased, the penalty on the corresponding state is decreased. This makes it convenient to use $R$ for tuning with appropriate trade-offs between control {effort} and error decay.

Motivated by the previous example, we consider the following nonlinear optimal control problem
\begin{subequations}
\label{eq: opt-prob}
\begin{align}
\min_{u}\max_w & \int_0^\infty  q(x,R,S)+ \,  \norm{u}^2_R-\xi \norm{w}^2_S \, \dd s,\\
\text{s.t. }& \dot x= f(x)+ G^\top(x)\,u+ \overline G^\top(x)\,w,
\label{eq: sys-dyn}
\\
x(0)& =x_0. \nonumber
\end{align}
\end{subequations}
Here, $x\in\real^n$ denotes the state vector, $x(0)=x_0$ is the initial state and $f(x)$ is a nonlinear vector field representing  a mapping from $\real^n$ to $\real^n$. We assume that $f(x)$ is continuous and locally Lipschitz with $f(0)=0$, that is the zero state is a steady state when no inputs are applied. The input matrix $G(x)=[g^\top_1(x), \dots, g^\top_m(x)]^\top\in\real^{m\times n}$ is given by the nonlinear functions $g_i(x),\, i=1,\dots m$ that are  mappings from $\real^n$ to $\real^n$ and continuous over $\real^n$.  The disturbance input matrix $\overline G(x)=[\overline g^\top_{1}(x), \dots, \overline g^\top_{n_w}(x)]^\top \in \real^{{n_w} \times n }$ and given by the nonlinear functions $\overline g_{i}(x),\, i=1,\dots n_w$, that are mappings from $\real^n$ to $\real^n$ and continuous over $\real^n$.  
We denote by $w\in\real^{n_w}$ an unknown disturbance, $\xi$ is a positive constant, $R=R^\top>0,\,S =S^\top>0$ are design matrices. 
Moreover, the mapping~ $q: \real^n\to\real_{>0}$ vanishes only at the origin, that is $q(0)=0$ and will be determined in the remainder.

Our goal is to find a state feedback controller $u^*(x)\in\real^m$ that solves the following Hamiltonian-Jacobi-Isaacs-Equation (HJIE) to optimality.
\begin{align}
\label{eq: hjb-pde}
\!\!\!\!\!\!\min_{u}{ \max_{w}}\left\{L(x,u,w, R,S)+ \nabla^\top_x V \left(f(x)+\; G^\top(x)\, u\right)\right\}=0,
\end{align}
where $L(x,u,w,R,S)= { q(x,R,S)}+ \,  \norm{u}^2_R-\xi \norm{w}^2_S$, and $V: \real^n \mapsto \real_{> 0}$ is a the value function of the optimal control problem, defined as \cite[Ch.2]{bacsar2008h} $$V (x_0) := \inf_u\sup_w \int_0^\infty  \left(q(x,R,S)+ \,  \norm{u}^2_R-\xi \norm{w}^2_S \right) \dd s. $$

Throughout this work, we illustrate feedback control synthesis via cost design and using a  {\em control Lyapunov} function \cite{Sontag1999}, i.e.,  a Lyapunov function for the closed-loop system associated with some choice of the control law.

\subsection{Cost design for optimal control}
\label{subsec: opt-ctrl}
We start our analysis with the nominal optimal control problem \eqref{eq: opt-prob} and set $w=0$.
In the subsequent analysis, we propose an approach to solve the nonlinear control problem \eqref{eq: opt-prob} to optimality with an appropriate choice of the function $q(x,R)$ {in} the following theorem.
\begin{theorem} 
\label{thm: H2-ctrl}
Consider the nominal optimal control problem~\eqref{eq: opt-prob}, i.e., when $w=0$.
Let $V:\mathds{R}^n\mapsto \mathds{R}_{{ >} 0}$ be a continuously differentiable function associated with a stabilizing feedback control law 
\begin{align}
\label{eq: H2-opt-cost-1}
u^*(x,R)&= -\frac{1}{2}\, R^{-1} G(x)\, \nabla_x V,
\end{align}
 
where, \begin{align}
\label{eq: q-pos}
\nabla_x V^\top\left(f(x) {+} G^\top(x)\, u^*(x,R)\right)<  -\norm{u^{*}(x,R)}^2_R.   \end{align}

Define
\begin{align}
\label{eq: q-x}
  q(x,R)=-\nabla_x V^\top\left(f(x){+} G^\top(x) u^*(x,R)\right) -\norm{u^{*}(x,R)}^2_R.
\end{align} Then, the following statements hold: 
\begin{enumerate}
\item The unique optimal control is given by $u^*$ in \eqref{eq: H2-opt-cost-1}.
\item The optimal control problem~\eqref{eq: opt-prob} {has the optimal value $V(x_0)$}. 
\end{enumerate}
\end{theorem}

\begin{proof}
 Consider the Hamiltonian function $$H(x,u,\lambda)=L (x,u)+\lambda^\top(f(x)+G^\top(x)\,u),$$ where $\lambda\in\real^n$ is the vector of co-state variables. We minimize $H(x,\lambda)$ by calculating,
\begin{align*}
\frac{{\partial} H(x,u,\lambda)}{{\partial} u}&= 2\,  R  \, u^*(x)+\; G(x)\, \lambda. 
\end{align*}
The optimal controller reads as, 
\begin{align*}
 u^*(x)&= -\frac{1}{2} R^{-1} G(x)\, \lambda= - \frac{1}{2} R ^{-1} G(x)\nabla_x V,
\end{align*}
where we set $\lambda=\nabla_x V$, following { \cite[Ch.1.4]{vinter2010optimal}}. This coincides with the stabilizing controller \eqref{eq: H2-opt-cost-1}.

For the sufficiency for optimality {of \eqref{eq: H2-opt-cost-1}}, we plug-in the controller \eqref{eq: H2-opt-cost-1} into \eqref{eq: hjb-pde} and obtain,
\begin{align*}
&  q(x,R) -\norm{G(x) \nabla_x V}^2_{R^{-1} } +\nabla^\top_x V f(x)=0. 
\end{align*}
By choice of the function $q(x,R)$ in \eqref{eq: q-x}, the HJBE is satisfied. {The positive definiteness of $q(x,R)$ follows from the inequality \eqref{eq: q-pos}}. 
We conclude that $V$ is a value function and the control law \eqref{eq: H2-opt-cost-1} is sufficient for optimality.
The optimal value is given by $V(x_0)$ and the proof is standard. See e.g. \cite[Ch 5.]{liberzon2011calculus} 



\end{proof}

\begin{remark}
We make the following observations:
\begin{itemize}
\item 
The inequality \eqref{eq: q-pos} is equivalent to, $$\dot V(x)< -\norm{u^{*}(x)}^2_R.$$ This implies by Lyapunov's second method that the origin is asymptotically stable for all system trajectories in closed-loop with \eqref{eq: H2-opt-cost-1}. 
\item Our approach relies on feedback design via a {control Lyapunov function} $V(x)$ to find a stabilizing controller $u^*(x)$ of the form \eqref{eq: H2-opt-cost-1}. By cost design of $q(x,R)$ as defined in \eqref{eq: q-x}, $V(x)$ is a value function of \eqref{eq: opt-prob} and we recover the optimal controller \eqref{eq: H2-opt-cost-1}.
\item Given a {control Lyapunov function} $V$, the matrix $R>0$ represents a tuning knob that can be used to improve the error decay or minimize the control effort. Note that $V$ is a value function of the optimal control problem~\eqref{eq: opt-prob} with any positive definite matrix $R'$, where ${R}' \leq R$ and associated with the cost function $L(\cdot, R')$ given in \eqref{eq: opt-prob}.
\item The cost {design} in \eqref{eq: q-x} exploits the intrinsic properties of the origin of the {open-loop} or {unforced} system \eqref{eq: sys-dyn} (i.e., when $u=0$) to achieve optimality. In particular, if $\nabla^\top_x V f(x)< 0$, then the inequality \eqref{eq: q-pos} is always satisfied (for any positive definite $R$) and the origin of the {unforced} system is asymptotically stable with the Lyapunov function $V(x)$. In this case, the matrix $R>0$ can be tuned {\em arbitrarily} {with} the same fixed $V(x)$. 
\end{itemize}
\end{remark}

{\em Example 2 (Linear systems)}
Consider the following LTI system together with $q(x,R)=x^\top Q{ (R)}\, x$, where $Q(  R)\in\real^{n\times n}$ is a matrix to be determined {with $R=R^\top>0$}.
\begin{align}
\dot x= A\,x+ B\,u,\ x(0)=x_0,
\label{eq: LTI}
\end{align}
where $A\in\real^{n\times n},\, B\in\real^{n\times m}, \; u\in\real^m$ and $x_0\in\real^n$. Given the Lyapunov function defined by $$V(x)= \frac{1}{2}\, x^\top P \,x, \,P  =P^\top>0,$$we apply Theorem \ref{thm: H2-ctrl} and the optimal controller is given by,
\begin{align}
\label{eq: lin-opt-ctrl-cost}
u^*(x, { R})=-\frac{1}{2} R  ^{-1} B^\top P\, x.    
\end{align}
We demonstrate in the sequel, that the application of optimal control theory is simplified, if we keep $P$ fixed and only tune the matrices $R$ and consequently $Q(R)$ given as in \eqref{eq: q-x} by,
\begin{align}
\label{eq: Q}
 { Q(R)}= \frac{1}{4\,} P\, B \,R  ^{-1}B^\top P-A^\top\, P - P\,A.
\end{align}  


 { Given a positive definite $ { Q}$ defined in \eqref{eq: Q}, the matrix $R$} can be tuned by choice of {any} positive definite matrices $R'\leq~R$  { with $Q({R'})$ in \eqref{eq: Q}}. Thus, we do not need to resolve the algebraic Riccati equation \eqref{eq: Q} for every value of the input matrix $R$, while fixing the positive definite matrix $P$. 

{\em Special case:} Under the assumption that $A$ is asymptotically stable, let $P>0$  satisfy,
\begin{align}
\label{eq: LE}
P\,A+A^\top P=-Q^*,\, Q^* =Q^{*\top}>0.    
\end{align}
Then, the matrix $Q(R')$ in \eqref{eq: Q} is a positive definite matrix for any other positive definite matrix $R'>0$. The resulting control law \eqref{eq: lin-opt-ctrl-cost} is optimal using the matrix $P$ in \eqref{eq: LE}.

The following illustrative example {is taken from} \cite{jouini2021optimal}.

{\em Example 3 (no dynamics):} 
Consider the optimal control problem described by,
\begin{align}
\min\limits_u& \int_0^\infty q (x(s))+\norm{u(s)}^2_R \;\dd s,\, R=R^\top>0,\\
 & \dot x= u,  \quad  { x(0)=x_0,} \nonumber
\end{align}
where $x\in\real^n$ is the state vector, $u\in\real^n$ is the control input and the mapping $q(x,R)$ is to be determined. Given a  { continuously} differentiable function $V(x)>0$ with $V(0)=0$, we arrive at the optimal feedback controller, 
\begin{align}
\label{eq: no-dyn-opt-ctrl}
u^*(x, { R})=-\frac{1}{2} R^{-1}\, \nabla_x V, 
\end{align}
associated with the cost function given by Theorem \ref{thm: H2-ctrl} as $$q(x, { R})=\frac{1}{4}\norm{\nabla_x V}^2_{R^{-1}}.$$ 
Observe that, due to the {trivial} system dynamics, i.e., $f(x)=0$, we can select any other control input matrix $R'>0$ {with $L(x,{R'})$}, while assuring optimality  { of $u^*(x,R')$ in \eqref{eq: no-dyn-opt-ctrl}}.

\subsection{Cost design for $\mathcal{H}_\infty-$control}
\label{subsec: H-infty-ctrl}
 { We now turn our attention to the disturbed/robust optimal control problem \eqref{eq: opt-prob} by setting $w\neq 0$. We arrive to the following result.}


\begin{proposition} 
\label{prop: H-infy}
Consider the  { robust optimal control problem \eqref{eq: opt-ctrl}} together with continuously differentiable function $V:\mathds{R}^n\mapsto \mathds{R}_{ > 0}$ associated with a controller $u^*$ in \eqref{eq: H2-opt-cost-1}.  { Let $w^*(x)=\frac{1}{2\,\xi}S^{-1} \overline G(x) \nabla_x V$}. Assume that,
{ \begin{align}
\label{eq: qw-pos}
 &\nabla_x V^\top\left(f(x) { \,+\,}G^\top(x)\, u^*(x)+  \overline G^\top(x)\, w^*(x)\right)\\ 
 &<-\norm{u^{*}(x)}^2_R+\xi\, \norm{w^{*}(x)}^2_S, \nonumber
\end{align}}
and define
\begin{align}
\label{eq: H-infty-type-1}
q(x, { R,S)}=& -\nabla_x V^\top \left(f(x) { \,+\,}G^\top(x)\, u^*(x)+  \overline G^\top(x)\, w^*(x)\right)\\
&-\norm{u^{*}(x)}^2_R+\xi\, \norm{w^{*}(x)}^2_S. \nonumber
\end{align}
Then, 
\begin{enumerate}
\item The optimal control $u^*$ is given by \eqref{eq: H2-opt-cost-1}. 
\item The {robust} optimal control problem \eqref{eq: opt-ctrl} {has the optimal value $V(x_0)$}.  
\end{enumerate}
\end{proposition}

\begin{proof}
For $w=0$, the optimal controller is given by \eqref{eq: H2-opt-cost-1}.
For $u=0$, we determine the worst case disturbance $w=w^*$, i.e., that maximizes the Hamiltonian function, $$H(x,u,\nabla_x V)=\max_w \{L (x, u, w)+\nabla^\top_x V(f(x)+ G^\top_w(x)\, w)\}.$$ This is achieved at $w=w^*$, where
\begin{align*}
-2\, \xi\, S w^*+  \overline G(x)  \nabla_x V=0, 
\end{align*}
which in turn implies that, 
\begin{align}
\label{eq: worst-w}  
w^*(x)=\frac{1}{2\,\xi}S^{-1} \overline G(x) \nabla_x V.
\end{align} 

Next, we plug in \eqref{eq: H2-opt-cost-1} into HJIE \eqref{eq: hjb-pde} and obtain,
\begin{align*}
&  q(x, { R,S)}+\nabla^\top_x V \left(f(x)-\frac{1}{4}G^\top(x)R^{-1} G(x)\nabla_x V\right)\\
&+\nabla^\top_x V \overline G^\top (x)\, w^*(x)-\xi \norm{w^{*}(x)}^2_S=0.\end{align*}    

By  { letting} $w^*(x)$ as in \eqref{eq: worst-w}, we arrive at the function $q(x, { R,S)}$ in \eqref{eq: H-infty-type-1} and the HJIE in \eqref{eq: hjb-pde} is satisfied.  { The positive definiteness of $q(x,R,S)$ is guaranteed by \eqref{eq: qw-pos}.}
 { This shows that $V$ is a value function of the robust optimal control problem \eqref{eq: opt-ctrl}.
  The optimal value is given by $V(x_0)$ and the proof is standard. See e.g. \cite[Thm 4.15]{bacsar2008h}}. 



\end{proof}

\begin{remark}
We have the following observations:
\begin{itemize}
 \item The system in closed-loop with \eqref{eq: H2-opt-cost-1} is finite-gain $\mathcal{L}_2-$ stable with $\mathcal{L}_2$ gain less than or equal to $2\sqrt{\xi}$.
\item For a given value function $V(x)$, the design matrices $R>~0$  and $S>0$ are tuning knobs that can be exploited to penalize the control input and disturbance deviations with the same $V$ {and any positive definite matrices $R'$ and $S'$ with $R'\leq R$, $S'\geq S$ and $L(\cdot, R', S')$ in \eqref{eq: opt-prob}.}
\item If it holds that, $$\nabla^\top_x V (f(x) {+\, \overline G^\top}(x)\, w^*(x))< 0,$$
 then, the origin is asymptotically stable for the worst case disturbance $w^*(x)$ and $V(x)$ is a Lyapunov function of the  { unforced} system. Thus,  {condition \eqref{eq: qw-pos} is always satisfied and $q(x,R',S')$} in \eqref{eq: Q-w} is positive {definite} independently of the choice of $R'$ and $S'$ and we can tune these design matrices arbitrarily using the same fixed $V$.
\end{itemize}
\end{remark}
We illustrate our approach using the following example.

{\em Example 4 (Linear systems)}
Given the LTI system, 
\begin{align}
\label{eq: sys-Hinfty-w}
 \dot x&= A\, x+ B\, u+ { \overline B} \, w, \quad  { x(0)=x_0} 
\end{align}
where $\overline B\in\real^{n\times n_w} $ is disturbance input matrix and $w\in\real^{n_w}$ is unknown additive disturbance. 
We define the cost function, 
\begin{align}
\label{eq: linear-cost}
L(x,u, w, R,S)= \norm{x}^2_Q+ \,  \norm{u}^2_R-\xi\, \norm{w}^2_S,\, \xi>0.    
\end{align} 
Following Proposition \ref{prop: H-infy}, we select
\begin{align}
\label{eq: Q-w}
\!\!\! { Q(R,S)}=\frac{1}{4\, } P  B R^{-1} B^\top P-\frac{1}{4 \xi} P \overline B S^{-1}\overline B^\top P-P\, A-A^\top P. 
\end{align} 
Given a positive definite matrix $P$, so that $Q>0$, where $Q$ is given in \eqref{eq: Q-w}. Then we can tune the design matrices $S$ and $R$ by choice of positive definite matrices $R'$ and $S'$ with $R'{ \leq} R$ and $ S' { \geq} S$ using the same matrix $P$ with  { $L(\cdot, R',S')$ in \eqref{eq: linear-cost}}. 

{\em Special case:} Under the assumption that $A$ is asymptotically stable, given a positive definite solution $P=K^{-1}$ where,
$$A\,K+K\, A^\top+\frac{1}{4\, \xi}  \overline B^\top\, S^{-1}\, \overline B < 0,$$
then $Q(R',S')>0$ as given in \eqref{eq: Q-w} and for any other positive definite matrices $R'$ and $S'$, the control law \eqref{eq: lin-opt-ctrl-cost} is optimal using the same matrix $P$ with  { $L(\cdot, R',S')$ in \eqref{eq: linear-cost}}.



\section{Application} 
\label{sec: application}
\subsection{Optimal control of coupled oscillators}
Consider a network of $n-$coupled oscillators whose $i-$th oscillator dynamics are described by the following differential equations.
\begin{align}
\label{eq: freq-droop}
\dot \theta_i&= \omega_i, {\qquad  i=1\dots n,} \\
M_i \dot \omega_i&=- D_i\, \omega_i- \sum_{j\in\mathcal{N}_i} b_{ij} \left( \sin(\theta_{ij})- \sin(\theta_{ij}^*)\right), \nonumber
\end{align}
with $M_i>0$ and $D_i>0$ and $b_{ij}>0$ denotes the coupling strength between the oscillators $i$ and $j$. 
Each oscillator is represented by its phase angle $\theta_i\in\real$ and frequency $\omega_i\in\real$.
Let $\omega=[\omega_1 ,\dots, \omega_n]^\top$, $\theta=[\theta_1 ,\dots,\theta_n]^\top$ and $\theta^*=[\theta^*_1, \dots, \theta^*_n]^\top$ be the vector of the {relative (to a nominal)} oscillator frequencies, oscillator angles and {nominal} steady state angles respectively. 
Define $\theta_{ij}=\theta_i-\theta_j$ and $\theta^*_{ij}=\theta^*_i-\theta^*_j$. Let $\mathcal{B}$ be the incidence matrix of the underlying graph $G$. 

Given a trajectory  $[\theta(t)^\top, \omega(t)^\top]^\top$ of \eqref{eq: freq-droop}, then $[(\theta(t)+\alpha \mathds{1}_n)^\top, \omega(t)^\top]^\top, \, \alpha\in\real$ is also a trajectory of the system \eqref{eq: freq-droop}. To eliminate this rotational invariance, we consider the following  coordinate transformation,
\begin{align}
\label{eq: coord-trafo}
\delta(t)= \mathcal{B}^\top \theta(t) \in\real^{m}.
\end{align}
Let $\theta^s$ be an induced steady state angle of \eqref{eq: opt-prob-droop} with steady state frequency $\omega^*=0$, $\delta^s=\mathcal{B}^\top\theta^s$ and  $\delta^*=\mathcal{B}^\top \theta^*\in\real^m$ be the nominal angle differences.
Observe that local asymptotic stability of $[\delta^{s\top}, 0^\top]^\top$ is equivalent to local asymptotic convergence of the solutions of \eqref{eq: freq-droop}  to $[\theta^{s\top}, 0^\top]^\top$. See for e.g.{ \cite{monshizadeh2017stability}}. {Next}, we make the following assumption. 
\begin{assumption}[\cite{monshizadeh2017stability}]
\label{ass: bounded-sol}
Assume that the steady state vector $\delta^s\in\real^m$ satisfies, 
\begin{align*}
\mathcal{B} \;\Xi \;\sin(\delta^s)=\mathcal{B} \;\Xi\sin(\delta^*),
\end{align*}
for all $\delta^s\in \text{Im}(B^\top)\cap (-\frac{\pi}{2}, \frac{\pi}{2})^m$.
\end{assumption}
Next, {consider} the following optimization problem, 
\begin{align}
 \label{eq: opt-prob-droop}
\min\limits_u& \int_0^\infty q(\delta(s), \omega(s))+\norm{u(s)}^2_R\,-\xi\, \norm{w(s)}^2_S  \; \dd s\nonumber\\
\mathrm{s.t. }&\quad \dot \delta= \mathcal{B}^\top\omega+ u,\nonumber \\
&M \dot \omega=- D\, \omega- \mathcal{B}\, \Xi \,\left(\underline{\sin} (\delta)-\underline{\sin} (\delta^*)\right)+w,\\
  &(\delta(0),\omega(0))=(\delta_0,\omega_0),\nonumber
\end{align} 
where $M>0$ and $D>0$ are diagonal matrices of inertia and damping coefficients and the coupling strengths $b_{ij}>0$ are collected in the diagonal matrix $\Xi=\mathrm{diag}(b_{ij})$. Let $\xi$ be a positive constant and $R =R^\top$ and $S =S^\top$ be positive definite matrices, $u=[u_1, \dots, u_m]^\top\in\real^m$ be the input and $w=[w_1,\dots, w_n]^\top\in\real^n$ the disturbance vector.
Furthermore, consider the following function (see e.g. \cite{monshizadeh2017stability, dorfler2012synchronization}) given by,
\begin{align}
\!\!V(\delta{ -\delta^ { s}}, \omega)=&\frac{1}{2} \norm{\omega}^2_M \,-\mathds{1}^\top_n\, \, \Xi \,\, (\underline{\cos}(\delta)-\underline{\cos}(\delta^ { s}))\nonumber \\
&-\,\, (\delta-\delta^ { s})^\top \, \Xi \,\,\underline{\sin}(\delta^ { s}).\!\!
\label{eq: LF}
\end{align}
 { It is noteworthy that under Assumption \ref{ass: bounded-sol}, $V(\delta-\delta^s,\omega)$ in \eqref{eq: LF} is locally (i.e., in a neighborhood $\Omega$ of $(\delta^s,0)$) positive definite.
}
Next, we have the following corollary.
\begin{corollary}
Consider the optimal control problem \eqref{eq: opt-prob-droop} under Assumption \ref{ass: bounded-sol}. The value function $V(\delta{ -\delta^s},\omega)$ given by \eqref{eq: LF} satisfies the HJBE \eqref{eq: hjb-pde} together with the following formulas for the cost functions.
\begin{enumerate}
\item For $w=0$, then
\begin{align*}
q(\delta, \omega,   R)=&\frac{1}{4} \norm{\underline{\sin}(\delta)- \underline{\sin}(\delta^ { s})}_{\Xi\; R^{-1}\; \Xi}^2+ \norm{\omega}^2_D.
\end{align*}
\item For $w\neq 0$, if $D-\frac{1}{4\,\xi}S^{-1}> 0$, then
\begin{align*}
q(\delta, \omega,   R,   S)=&\frac{1}{4}\norm{\underline{\sin}(\delta)- \underline{\sin}(\delta^ { s})}_{\Xi\; R^{-1}\; \Xi}^2+ \norm{\omega}^2_{D-\frac{1}{4\xi}S^{-1}}. 
\end{align*}
\end{enumerate}
Moreover, the optimal controller is uniquely given by
\begin{align}
\label{eq: opt-ctrl}
u^*(\delta{,  R})=-\frac{1}{2}\, R^{-1}\,\Xi\,(\underline{\sin}(\delta)- \underline{\sin}(\delta^ { s})).
\end{align}
\end{corollary}
\begin{proof}
The two statements follow directly from Theorem \ref{thm: H2-ctrl} and Proposition \ref{prop: H-infy} with the Lyapunov function \eqref{eq: LF}.
{To see this,} Lie derivative of $V$ is given by $$\dot V(\delta{-\delta^s}, \omega)= -\norm{\omega}^2_D\,  \leq 0.$$
 Under Assumption \ref{ass: bounded-sol}, the sub-level sets of $V$ are bounded in a neighborhood $\Omega$ of $[\delta^{s\top}, 0^\top]^\top$. By applying Lasalle's invariance principle \cite{K02}, the trajectories of the dynamical system \eqref{eq: opt-prob-droop}  { starting at $\Omega$} converge to the set where $\omega=0$, which in turn implies that $\delta=\delta^s$, where $\delta^s-\delta^*$ is a constant angle vector. This establishes that $[\delta^{s\top}, 0^\top]^\top$ is locally asymptotically stable and $V$ in \eqref{eq: LF} is a  Lyapunov function for the system dynamics \eqref{eq: opt-prob-droop}, for all  { $x\in\Omega$}. For the second statement, the condition $D> \frac{1}{4\xi}S^{-1}$
ensures that $q(\cdot,  { R,  S})>0$ as in Proposition \ref{prop: H-infy}.
\end{proof}

Note that the controller $u^*(\delta{,  R})$ in \eqref{eq: opt-ctrl} is  { locally optimal, i.e., valid in a neighborhood $\Omega$ of $[\delta^{s\top}, 0^\top]^\top$} and distributed, i.e., depends on the angle differences of the neighboring oscillator angles and the functions $  q(\cdot, R')$ and $  q(\cdot, R', S')$  { remain} positive for any other positive definite matrices $R',   S'>0$.

\subsection{Simulations} 
\begin{figure}[ht!]
    \centering
    \includegraphics[scale=0.25, trim= 1.5cm 0cm 0cm 0cm]{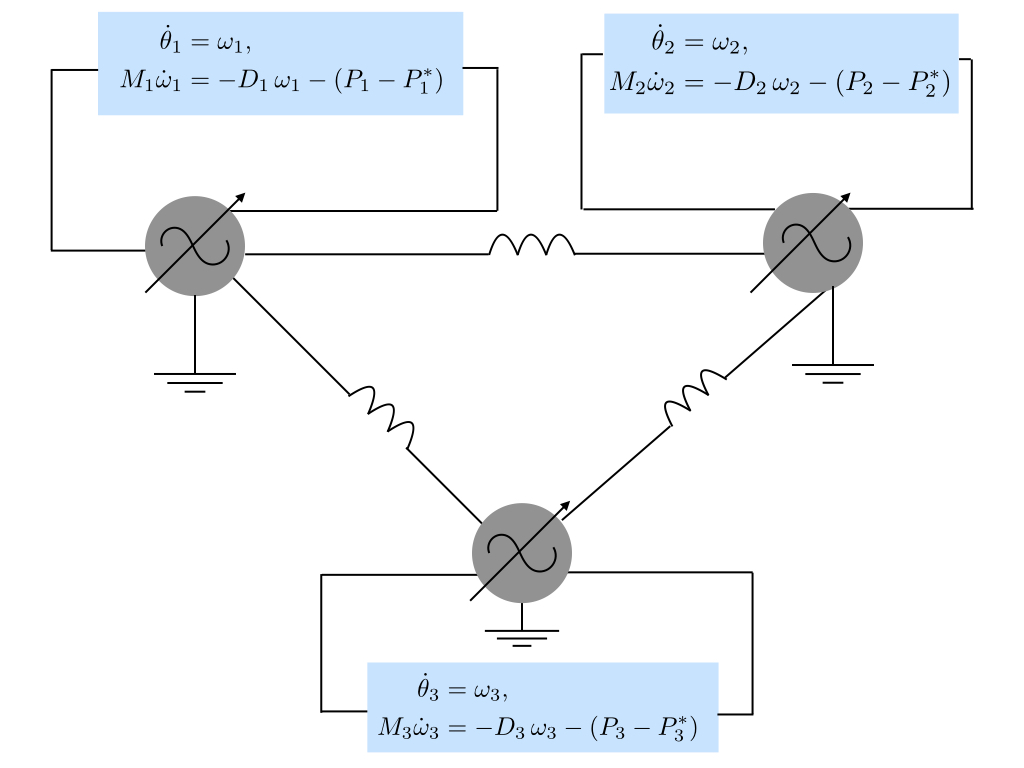}
    \caption{Three inverter system with dynamics given in \eqref{eq: freq-droop}, where $P_{i}=\sum_{j\in\mathcal{N}_i} b_{ij} \sin(\theta_{ij})$ and $P^*_i=\sum_{j\in\mathcal{N}_i} b_{ij} \sin(\theta^*_{ij})$ for $i=1,2,3$.}
    \label{fig: test-case}
\end{figure}
We adopt the same setup as in \cite{jouini2021optimal} and consider a network of three  inverters with system dynamics \eqref{eq: freq-droop}. The parameters $M_i$ and $D_i$ represent inertia and damping coefficients. The inverters are connected by purely inductive transmission lines with line susceptance $b_{ij}>0$ as shown in Figure \ref{fig: test-case}. 
We test numerically the derived optimal controller \eqref{eq: opt-ctrl} for nominal $(w=0)$ and disturbance attenuation $(w\neq 0)$ settings. The disturbance $w=[w_1,\dots,w_n]^\top\in\real^n$ models for e.g. DC-side generation and AC side fluctuations \cite{kundur2004definition}. For simplicity, we set all line susceptances $b_{ij}$ to one per unit (p.u.). The parameters in \eqref{eq: freq-droop} are 
chosen uniformly with $M_1=M_2=M_3=0.01 \mathrm{[s^2/rad]}$ and $D_1=D_2=D_3=0.1 \mathrm{[s/rad]}$.

Time-domain simulations of the open-loop angle differences and frequencies of the three inverter system with the  { unforced} inverter system (i.e., $u=0$) in \eqref{eq: freq-droop} and the desired steady state angle differences $\delta^*=[0,0,0]^\top$, starting at $\delta(0)=[0.02; 0.015, 0, 0]$ show that $\delta^s=[0.0113,
0.0113, -0.0113]$ and thus satisfy Assumption~\ref{ass: bounded-sol}. Moreover, the inverters frequencies synchronize at $ \omega^*=0$.

Next, we consider the optimal control problem \eqref{eq: opt-prob-droop} and implement the control law \eqref{eq: opt-ctrl} both for nominal $(w=0)$ and disturbance attenuation  $(w\neq 0)$. We additionally verify the optimal controller for two examples of the design matrix $R_1$ and $R_2$. Once in closed-loop with the optimal controller  \eqref{eq: opt-ctrl}, all frequencies synchronize at nominal  with a decay towards zero and improved transient behavior both for $R_1=0.1 \cdot I_3$ and $R_2=0.01\cdot I_3$ in Figures \ref{fig: R1} and \ref{fig: R2} respectively. Compared to the input matrix $R_1$, the matrix $R_2$ penalizes less the input variations and thus allows for more control input effort leading to faster error decay rate.
In the presence of non-zero, additive and randomly generated disturbances $w=[w_3, w_2, w_1]^\top$, Figure \ref{fig: S} shows that the frequencies remain bounded, albeit non-synchronized, which is in accordance with our theory. The nominal and disturbed cost functions are decreasing towards a value that is nearby zero.


\begin{figure}[ht!]
    \centering
    \includegraphics[scale=0.5]{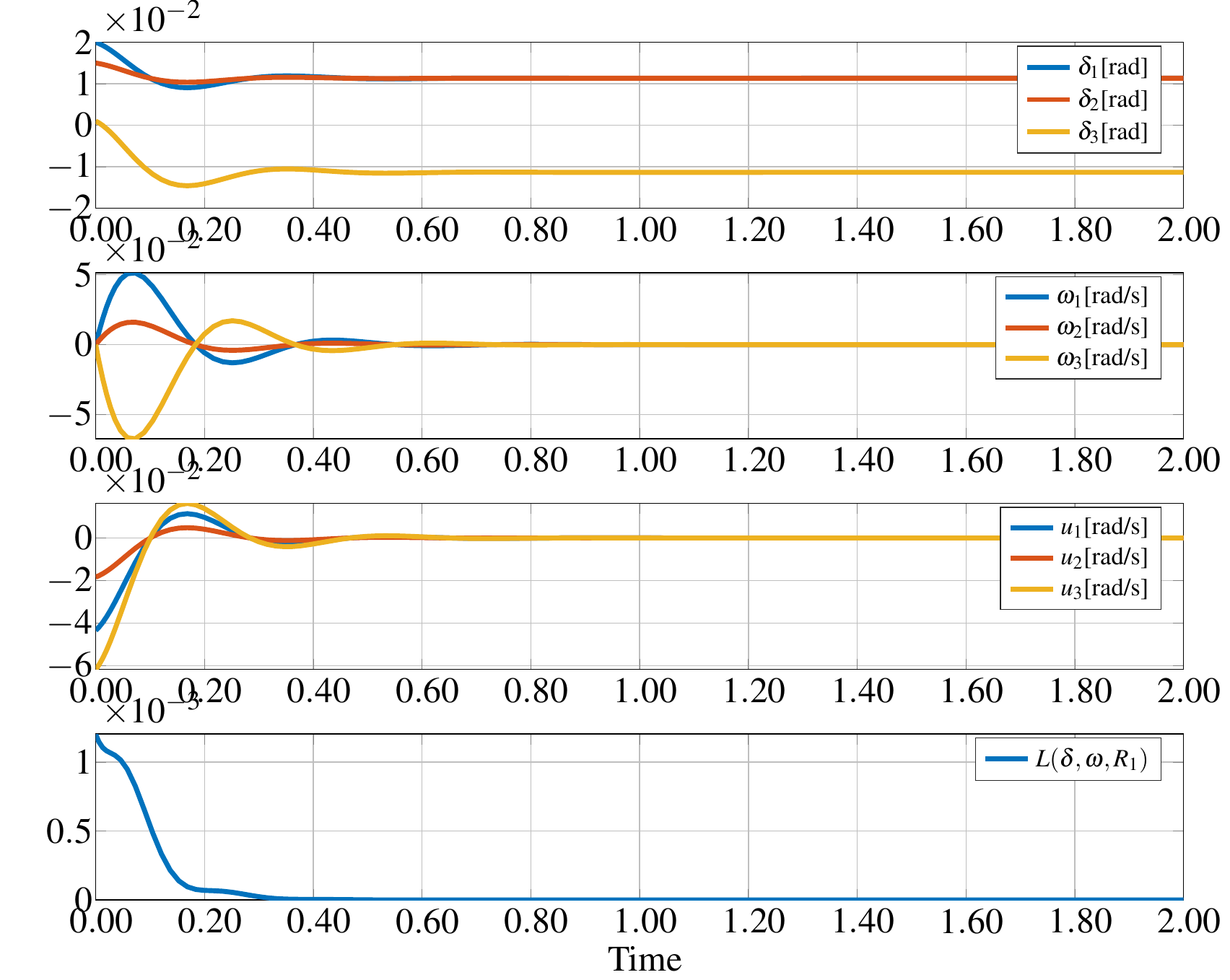}
    \caption{Simulations for $w=0$ of angle differences, frequencies, inputs and the cost function of the three inverter system described in Figure \ref{fig: test-case} for $w=0$ after closing the loop with the optimal control \eqref{eq: opt-ctrl} with $R_1=0.1\cdot I_3$. The angles are stabilized at the specified steady state and the frequencies synchronize and decay towards zero. The cost function $L(\delta, \omega, R_1)$  strictly decreases towards a nearby zero value.}
    \label{fig: R1}
\end{figure}

\begin{figure}[ht!]
    \centering
    \includegraphics[scale=0.5]{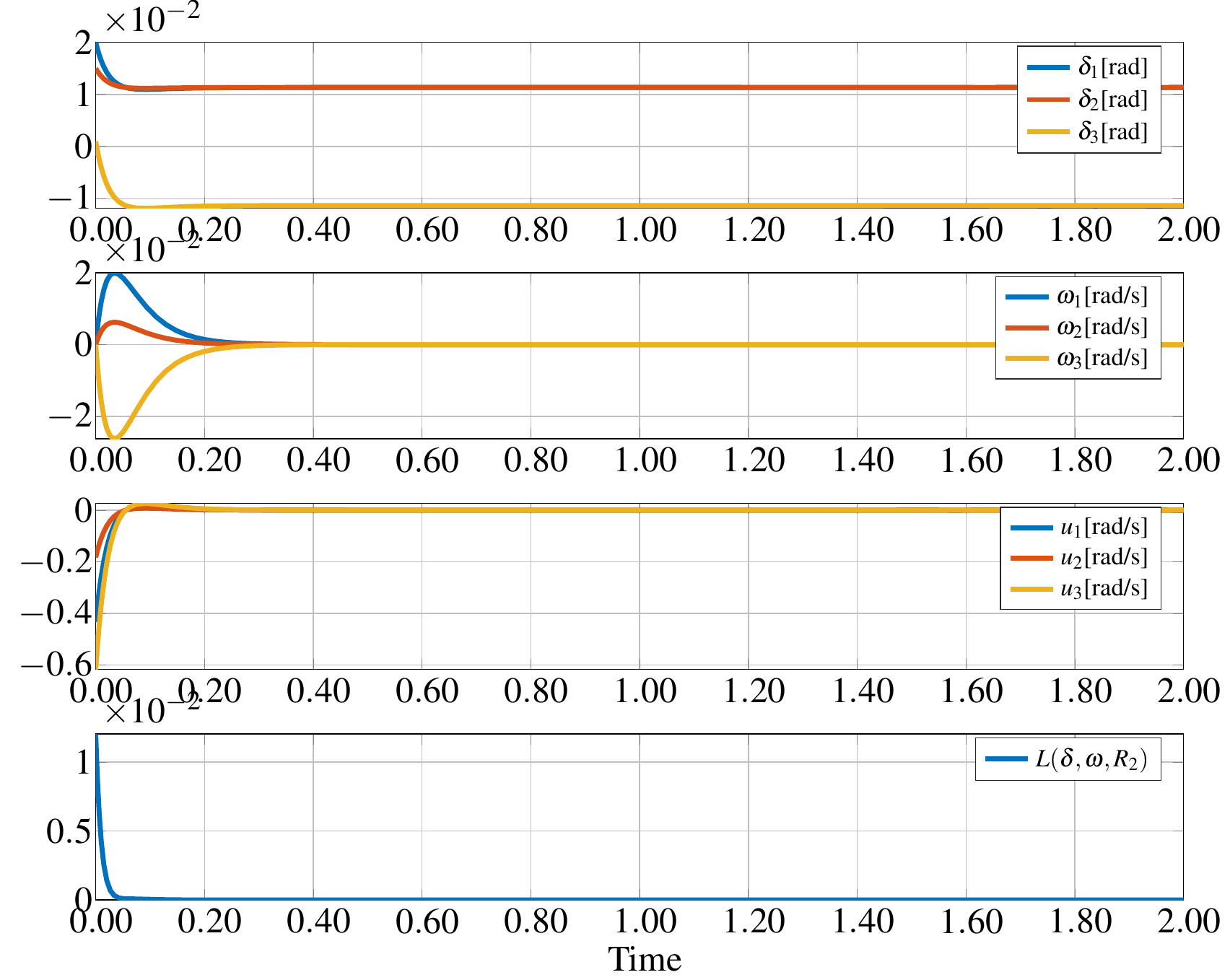}
    \caption{Simulations for $w=0$ of angle differences, frequencies, inputs and the cost function after closing the loop with the optimal control \eqref{eq: opt-ctrl} for $R_2=0.1\cdot R_1$. The error decay transients of the angles and frequencies improve significantly with more control effort (compared to Figure \ref{fig: R1}). The nominal cost function $L(\delta, \omega, R_2)$ decays towards a nearby zero value.}
    \label{fig: R2}
\end{figure}

\begin{figure}[ht!]
    \centering
    \includegraphics[scale=0.5]{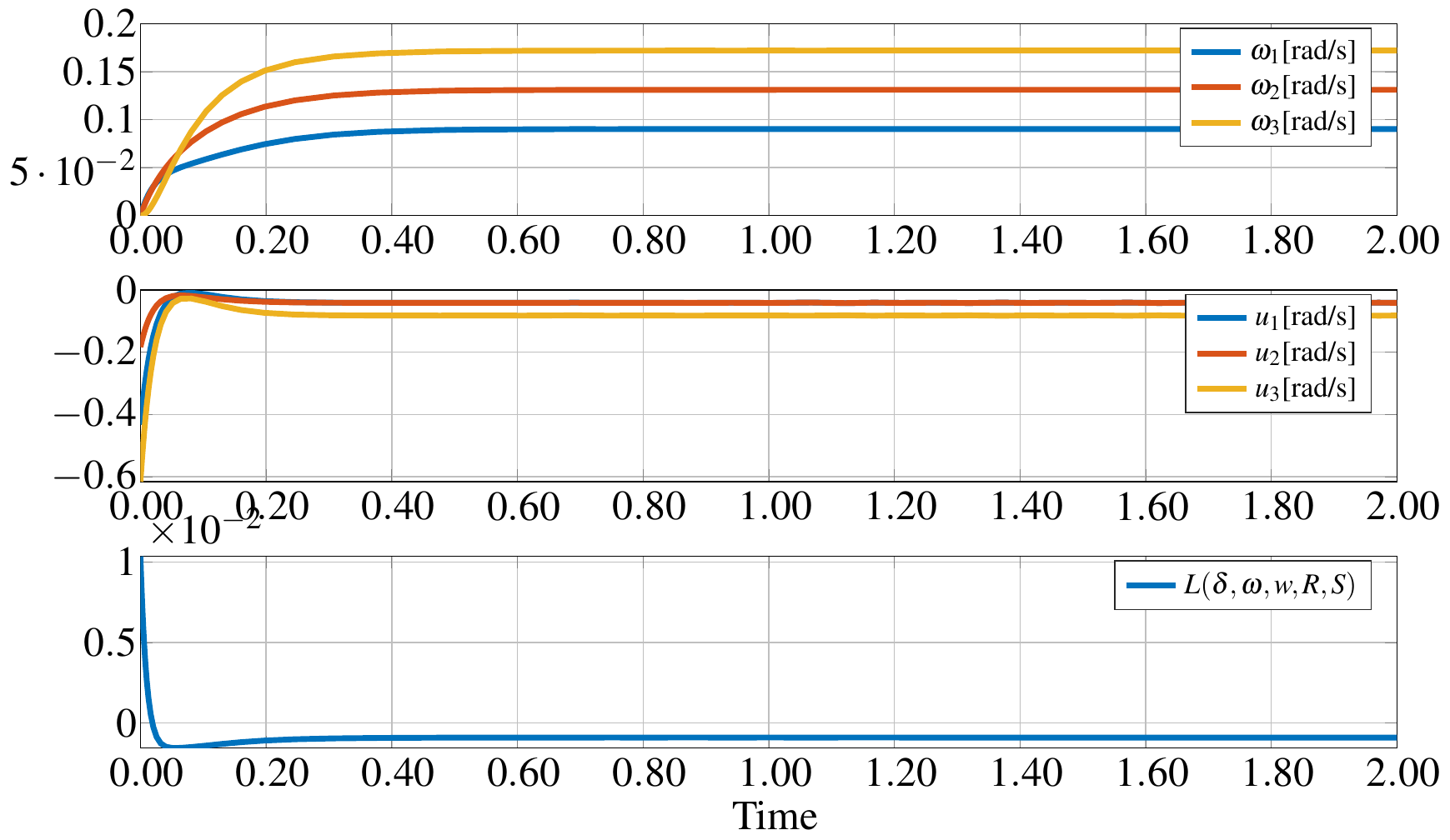}
    \caption{Simulations of the frequencies, input and cost function of the three-inverter system for constant non-zero disturbance $w$ in closed-loop with \eqref{eq: opt-ctrl} for $R=0.01\cdot I_3, \, S=I_3$ and $\xi=2.8>\frac{1}{4\,D}$. The frequencies remain bounded and the cost function $L(\delta,\omega,w,R,S)$ in \eqref{eq: opt-prob} takes negative values with a randomly generated disturbance $w\neq w^*(x)$.}
    \label{fig: S}
\end{figure}

\section{Conclusion}
We studied the role of cost design {for optimal feedback control} in satisfying HJBE  {or HJIE} in theory and via examples and an application to control of oscillatory systems. The optimal control problem reduces to a decision on how to tune the control gains, while the value function remains unchanged. The optimal controller is thus comparable to a linear quadratic regulator. It is in our future interest to investigate the ramifications of the proposed design method on the study of {passive systems and} constrained optimal control problems.

\bibliographystyle{IEEEtran}
\bibliography{root.bib}

\begin{thebibliography}{10}
\providecommand{\url}[1]{#1}
\csname url@rmstyle\endcsname
\providecommand{\newblock}{\relax}
\providecommand{\bibinfo}[2]{#2}
\providecommand\BIBentrySTDinterwordspacing{\spaceskip=0pt\relax}
\providecommand\BIBentryALTinterwordstretchfactor{4}
\providecommand\BIBentryALTinterwordspacing{\spaceskip=\fontdimen2\font plus
\BIBentryALTinterwordstretchfactor\fontdimen3\font minus
  \fontdimen4\font\relax}
\providecommand\BIBforeignlanguage[2]{{%
\expandafter\ifx\csname l@#1\endcsname\relax
\typeout{** WARNING: IEEEtran.bst: No hyphenation pattern has been}%
\typeout{** loaded for the language `#1'. Using the pattern for}%
\typeout{** the default language instead.}%
\else
\language=\csname l@#1\endcsname
\fi
#2}}

\bibitem{9186331}
M.~{Sassano} and A.~{Astolfi}, ``Combining {P}ontryagin's {P}rinciple and
  dynamic programming for linear and nonlinear systems,'' \emph{IEEE
  Transactions on Automatic Control}, vol.~65, no.~12, pp. 5312--5327, 2020.

\bibitem{bertsekas2011dynamic}
D.~P. Bertsekas, ``Dynamic programming and optimal control 3rd edition, volume
  ii,'' \emph{Belmont, MA: Athena Scientific}, 2011.

\bibitem{liberzon2011calculus}
D.~Liberzon, \emph{Calculus of variations and optimal control theory: a concise
  introduction}.\hskip 1em plus 0.5em minus 0.4em\relax Princeton University
  Press, 2011.

\bibitem{lukes1969optimal}
D.~L. Lukes, ``Optimal regulation of nonlinear dynamical systems,'' \emph{SIAM
  Journal on Control}, vol.~7, no.~1, pp. 75--100, 1969.

\bibitem{powell2007approximate}
W.~B. Powell, \emph{Approximate Dynamic Programming: Solving the curses of
  dimensionality}.\hskip 1em plus 0.5em minus 0.4em\relax John Wiley \& Sons,
  2007, vol. 703.

\bibitem{scherer2001theory}
C.~Scherer, ``Theory of robust control,'' \emph{Delft University of
  Technology}, pp. 1--160, 2001.

\bibitem{zhou1998essentials}
K.~Zhou and J.~C. Doyle, \emph{Essentials of robust control}.\hskip 1em plus
  0.5em minus 0.4em\relax Prentice hall Upper Saddle River, NJ, 1998, vol. 104.

\bibitem{bacsar2008h}
T.~Ba{\c{s}}ar and P.~Bernhard, \emph{H-infinity optimal control and related
  minimax design problems: a dynamic game approach}.\hskip 1em plus 0.5em minus
  0.4em\relax Springer Science \& Business Media, 2008.

\bibitem{jouini2021optimal}
T.~Jouini and E.~Tegling, ``Optimal control for power converters based on phase
  angle feedback,'' \emph{arXiv preprint arXiv:2101.11141}, 2021.

\bibitem{Sontag1999}
E.~D. Sontag, \emph{Control-Lyapunov functions}.\hskip 1em plus 0.5em minus
  0.4em\relax London: Springer London, 1999, pp. 211--216.

\bibitem{vinter2010optimal}
R.~Vinter, \emph{Optimal control}.\hskip 1em plus 0.5em minus 0.4em\relax
  Springer Science \& Business Media, 2010.

\bibitem{monshizadeh2017stability}
P.~Monshizadeh, C.~De~Persis, T.~Stegink, N.~Monshizadeh, and A.~van~der
  Schaft, ``Stability and frequency regulation of inverters with capacitive
  inertia,'' in \emph{2017 IEEE 56th Annual Conference on Decision and Control
  (CDC)}.\hskip 1em plus 0.5em minus 0.4em\relax IEEE, 2017, pp. 5696--5701.

\bibitem{dorfler2012synchronization}
F.~D\"orfler and F.~Bullo, ``Synchronization and transient stability in power
  networks and nonuniform kuramoto oscillators,'' \emph{SIAM Journal on Control
  and Optimization}, vol.~50, no.~3, pp. 1616--1642, 2012.

\bibitem{K02}
H.~K. Khalil, \emph{Nonlinear systems}, 3rd~ed.\hskip 1em plus 0.5em minus
  0.4em\relax Prentice hall New Jersey, 2002.

\bibitem{kundur2004definition}
P.~Kundur, J.~Paserba, V.~Ajjarapu, G.~Andersson, A.~Bose, C.~Canizares,
  N.~Hatziargyriou, D.~Hill, A.~Stankovic, C.~Taylor, \emph{et~al.},
  ``Definition and classification of power system stability,'' \emph{IEEE
  transactions on Power Systems}, vol.~19, no.~2, pp. 1387--1401, 2004.

\end{thebibliography}

\end{document}